\documentclass[10pt]{amsart}
\usepackage{amssymb,MnSymbol}
\usepackage{amsthm,amsmath}

\title[Symmetric approximations of pseudo-Boolean functions]{Symmetric approximations of pseudo-Boolean functions with applications to influence indexes}

\author{Jean-Luc Marichal}
\address{Mathematics Research Unit, FSTC, University of Luxembourg, 6, rue Coudenhove-Kalergi, L-1359 Luxembourg, Grand Duchy of Luxembourg.}
\email{jean-luc.marichal[at]uni.lu}

\author{Pierre Mathonet}
\address{Mathematics Research Unit, FSTC, University of Luxembourg, 6, rue Coudenhove-Kalergi, L-1359 Luxembourg, Grand Duchy of Luxembourg.}
\email{p.mathonet[at]ulg.ac.be}

\date{February 27, 2012}

\begin{document}

\theoremstyle{plain}
\newtheorem{theorem}{Theorem}%[section]% Supprimer [section] pour une numérotation linéaire
\newtheorem{lemma}[theorem]{Lemma}
\newtheorem{proposition}[theorem]{Proposition}
\newtheorem{corollary}[theorem]{Corollary}
\newtheorem{fact}[theorem]{Fact}
\newtheorem*{main}{Main Theorem}

\theoremstyle{definition}
\newtheorem{definition}[theorem]{Definition}
\newtheorem{example}[theorem]{Example}

\theoremstyle{remark}
\newtheorem*{conjecture}{\indent Conjecture}
\newtheorem{remark}{\indent Remark}
\newtheorem{claim}{Claim}

\newcommand{\N}{\mathbb{N}}
\newcommand{\R}{\mathbb{R}}
\newcommand{\I}{\mathbb{I}}
\newcommand{\B}{\mathbb{B}}
\newcommand{\Vspace}{\vspace{2ex}}
\newcommand{\bfx}{\mathbf{x}}
\newcommand{\bfy}{\mathbf{y}}
\newcommand{\bfh}{\mathbf{h}}
\newcommand{\bfe}{\mathbf{e}}
\newcommand{\Q}{Q}
\newcommand{\os}{\mathrm{os}}

\begin{abstract}
We introduce an index for measuring the influence of the $k$th smallest variable on a pseudo-Boolean function. This index is defined from a
weighted least squares approximation of the function by linear combinations of order statistic functions. We give explicit expressions for both
the index and the approximation and discuss some properties of the index. Finally, we show that this index subsumes the concept of system
signature in engineering reliability and that of cardinality index in decision making.
\end{abstract}



\maketitle

%---------------------------------------------------------------------------------------------- Section 1
\section{Introduction}

Boolean and pseudo-Boolean functions play a central role in various areas of applied mathematics such as cooperative game theory, engineering
reliability, and decision making (where fuzzy measures and fuzzy integrals are often used). In these areas indexes have been introduced to
measure the importance of a variable or its influence on the function under consideration (see, e.g., \cite{BouKahKalKatLin92,Mar00}). For
instance, the concept of importance of a player in a cooperative game has been studied in various papers on \emph{values} and \emph{power
indexes} starting from the pioneering works by Shapley~\cite{Sha53} and Banzhaf~\cite{Ban65}. These power indexes were rediscovered later in
system reliability theory as Barlow-Proschan and Birnbaum measures of importance (see, e.g., \cite{Ram90}).

In general there are many possible influence/importance indexes and they are rather simple and natural. For instance, a cooperative game on a
finite set $[n]=\{1,\ldots,n\}$ of \emph{players} is a set function $v\colon 2^{[n]}\to\R$ with $v(\varnothing)=0$, which associates with any
\emph{coalition} of players $S\subseteq [n]$ its \emph{worth} $v(S)$. The \emph{Banzhaf value} of player $i\in [n]$ in the game $v$ is then
defined as
\begin{equation}\label{eq:sa8ffd}
\phi_B(v,i) ~=~ \frac{1}{2^{n-1}}\sum_{S\subseteq [n]\setminus\{i\}}\big(v(S\cup\{i\})-v(S)\big) ~=~ \frac{1}{2^{n-1}}\sum_{S\ni
i}v(S)-\frac{1}{2^{n-1}}\sum_{S\not\ni i}v(S).
\end{equation}
Thus, $\phi_B(v,i)$ is the average of the marginal contributions of player $i$ to all coalitions $S\subseteq [n]\setminus\{i\}$, or the
difference between the average worth over all coalitions $S\ni i$ and the average worth over all coalitions $S\not\ni i$. Considering weighted
averages instead of symmetric averages gives rise to various \emph{probabilistic values} (see \cite{Web88}), including the Shapley value and
weighted Banzhaf values (see \cite{MarMat11a}).

The choice of a suitable influence/importance index depends on the practical problem under consideration and is usually made by considering the
properties that the index should satisfy. This is why many indexes have been characterized axiomatically. Besides these characterizations, it is well
known in statistics that one can measure the influence of a variable using linear regression. This approach was applied successfully to
pseudo-Boolean functions by Hammer and Holzman~\cite{HamHol92}, who showed that the Banzhaf value appears as the coefficients of the linear
terms of the standard least squares approximation of a game (or its corresponding pseudo-Boolean function) by a function of degree at most $1$.
Weighted versions of this least squares approach were also considered to characterize the Shapley value \cite{ChaGolKeaRou88} and weighted
Banzhaf values \cite{MarMat11a}.

Slightly different influence indexes emerged in certain applications where it is not the influence of a variable on a function that is to be
measured but rather the influence of adding a variable to a given subset of variables. For instance, considering a system made up of $n$
interconnected components with independent and identically distributed (i.i.d.) lifetimes, Samaniego \cite{Sam85,Sam07} defined the
\emph{signature} of the system as the $n$-tuple $(s_1,\ldots,s_n)$ where $s_k$ is the probability that the $k$th failure causes the system to
fail. Due to the i.i.d.\ assumption, the signature only depends on the (Boolean) structure function of the system. Thus, the number $s_k$ can be
interpreted as a measure of the influence on the structure function of adding a $k$th element to the set of failed components. Another example
of such a measure of influence was introduced by Yager~\cite{Yag02} in the context of fuzzy measures. Considering a fuzzy measure on an $n$-set
$X$, he introduced the \emph{cardinality index} as the $n$-tuple $(C_0,\ldots,C_{n-1})$, where $C_k$ is the average gain in certitude one gets
when going from a $k$-subset to a $(k+1)$-subset.

In this paper we show that the concepts of system signature and cardinality index are special instances of a more general notion of influence
index: the \emph{influence index of the $k$th smallest variable} on a pseudo-Boolean function. We define this index by considering the least
squares approximation of a given function by a linear combination of order statistic functions. Such linear combinations are particularly
suitable for encoding the influence that we want to measure and are exactly the symmetric (i.e., invariant under a permutation of the variables)
pseudo-Boolean functions (see Proposition~\ref{prop:asdf7}). Here we consider the general framework of \emph{arbitrarily weighted} least squares
approximations. In Section 2 we give explicit expressions for the approximation and discuss some of its properties. In Sections 3 and 4 we
introduce our influence index and show how it subsumes the concepts of system signature and cardinality index. We also show how this index can
be used in cooperative game theory to define a new influence index.

We employ the following notation throughout the paper. We denote by $\B$ the two-element set $\{0,1\}$. For any $\bfx\in\B^n$, we set
$|\bfx|=\sum_{i=1}^nx_i$. For any $S\subseteq [n]=\{1,\ldots,n\}$, we denote by $\mathbf{1}_S$ the $n$-tuple whose $i$th coordinate is $1$, if
$i\in S$, and $0$, otherwise (with the particular cases $\mathbf{0}=\mathbf{1}_{\varnothing}$ and $\mathbf{1}=\mathbf{1}_{[n]}$).

Through the usual identification of the elements of $\B^n$ with the subsets of $[n]$, a pseudo-Boolean function $f\colon\B^n\to\R$ can be
equivalently described by a set function $v_f\colon 2^{[n]}\to\R$. We simply write $v_f(S)=f(\mathbf{1}_S)$. To avoid cumbersome notation, we
henceforth use the same symbol to denote both a given pseudo-Boolean function and its underlying set function, thus writing $f\colon\B^n\to\R$
or $f\colon 2^{[n]}\to\R$ interchangeably.

Recall that if the $\B$-valued variables $x_1,\ldots,x_n$ are rearranged in ascending order of magnitude $x_{(1)}\leqslant\cdots\leqslant
x_{(n)}$, then $x_{(k)}$ is called the \emph{$k$th order statistic} and the function $\os_k\colon\B^n\to\B$, defined as $\os_k(\bfx)=x_{(k)}$,
is the \emph{$k$th order statistic function}. We then have $\os_k(\bfx)=1$, if $\sum_{i=1}^nx_i\geqslant n-k+1$, and $0$, otherwise. As a matter
of convenience, we also formally define $\os_0\equiv 0$ and $\os_{n+1}\equiv 1$. An \emph{$L$-statistic} function is a linear combination of the
functions $\os_1,\ldots,\os_n$ while a \emph{shifted $L$-statistic} function is a linear combination of the functions $\os_1,\ldots,\os_{n+1}$.

%---------------------------------------------------------------------------------------------- Section 2
\section{Symmetric approximations}

In this section we present and solve the problem of approximation of pseudo-Boolean functions by shifted $L$-statistic functions and discuss a few
properties of the approximations.

Recall that any $n$-ary pseudo-Boolean function $f$ can always be represented by a multilinear polynomial of degree at most $n$ (see
\cite{HamRud68}). More precisely, $f$ can always be written in the form
\begin{equation}\label{eq:pBfPF}
f(\bfx)=\sum_{S\subseteq [n]} f(S)\,\prod_{i\in S}x_i\,\prod_{i\in [n]\setminus S}(1-x_i).
\end{equation}
By expanding the second product, we see that this polynomial can be further simplified into
$$%\begin{equation}\label{eq:fMob}
f(\bfx)=\sum_{S\subseteq [n]} a_f(S)\,\prod_{i\in S}x_i\, ,
$$%\end{equation}
where the set function $a_f\colon 2^{[n]}\to\R$, called the \emph{M\"obius transform} of $f$, is defined by
$$
a_f(S)=\sum_{T\subseteq S} (-1)^{|S|-|T|}\, f(T).
$$

Denote by $F(\B^n)$ the vector space of $n$-ary pseudo-Boolean functions and by $F_S(\B^n)$ the subspace of symmetric $n$-ary pseudo-Boolean
functions. It is clear that a function $f\in F(\B^n)$ is symmetric if and only if it is \emph{cardinality-based}, i.e., it satisfies the
property $f(S)=f(T)$ for every $S,T\subseteq [n]$ such that $|S|=|T|$. Equivalently, there exists a unique function
$\overline{f}\colon\{0,1,\ldots,n\}\to\R$ such that $f(S)=\overline{f}(|S|)$.

The following proposition shows that the shifted $L$-statistic functions are precisely those pseudo-Boolean functions that are symmetric.

\begin{proposition}\label{prop:asdf7}
A pseudo-Boolean function is symmetric if and only if it is a shifted $L$-statistic function.
\end{proposition}

\begin{proof}
The class of $n$-ary shifted $L$-statistic functions is clearly a subspace of $F_S(\B^n)$. Since each of these spaces has dimension $n+1$, they
must coincide.
\end{proof}

Given a weight function $w\colon\B^n\to {\left]{0},{\infty}\right[}$ and a function $f\in F(\B^n)$, we define the \emph{best symmetric
approximation of $f$ with respect to $w$} as the unique function $f_L\in F_S(\B^n)$ that minimizes the weighted squared distance
$$
\|f-g\|^2=\sum_{\bfx\in\B^n}w(\bfx)\big(f(\bfx)-g(\bfx)\big)^2
$$
among all symmetric functions $g\in F_S(\B^n)$. Since $\|\cdot\|$ is the norm associated with the inner product $$\langle
f,g\rangle=\sum_{\bfx\in\B^n}w(\bfx)f(\bfx)g(\bfx),$$ the solution $f_L$ of this approximation problem exists and is uniquely determined by the
orthogonal projection of $f$ onto $F_S(\B^n)$. We then write $f_L=A(f)$.

We will henceforth assume (without loss of generality) that the weights are multiplicatively normalized so that $\sum_{\bfx\in\B^n}w(\bfx)=1$.
Although this assumption is not necessary for most of the results, it will enable us to interpret $w$ as a probability distribution and make use
of certain concepts in probability theory.

%If the weights $w(S)$ were multiplicatively normalized so that $\sum_{\bfx\in\B^n}w(\bfx)=1$ (which is not a restriction), they would define a
%probability distribution over $\B^n$. This motivates the following definition.

\begin{definition}\label{def:vf}
For every $f\in F(\B^n)$, define $\overline{f}\colon \{0,1,\ldots,n\}\to\R$ as
\begin{equation}\label{eq:76dfsd}
\overline{f}(s) ~=~ E(f(\bfx)\mid |\bfx|=s) ~=~ \frac{\sum_{|\bfx|=s}w(\bfx)\, f(\bfx)}{\sum_{|\bfx|=s}w(\bfx)}\, .\footnote{Although
$\overline{f}$ depends explicitly on $w$, we use this notation for it is consistent with that introduced for cardinality-based set functions.
Note also that, in the special case when the weight function $w$ is symmetric, $\overline{f}(s)$ clearly reduces to ${n\choose
s}^{-1}\sum_{|\bfx|=s}f(\bfx)$.}
\end{equation}
We also formally define $\overline{f}({-1})=0$.
\end{definition}

The next theorem gives an explicit expression for $A(f)$.

\begin{theorem}\label{thm:as897}
The best symmetric approximation of $f\in F(\B^n)$ is given by
\begin{equation}\label{eq:98s7}
A(f)=\sum_{j=1}^{n+1}c_j\,\mathrm{os}_j\, ,
\end{equation}
where $c_j = \overline{f}(n-j+1)-\overline{f}(n-j)$ for every $j\in [n+1]$.
\end{theorem}

\begin{proof}
Since $F_S(\B^n)$ is spanned by the $n+1$ functions $\mathrm{os}_1,\ldots,\mathrm{os}_n,\mathrm{os}_{n+1}$, the projection $A(f)$ is
characterized by the conditions
\begin{equation}\label{eq:9s0df2}
\langle f-A(f),\mathrm{os}_i\rangle=0 \qquad (i\in [n+1]),
\end{equation}
that is,
\begin{equation}\label{eq:9s0df}
\sum_{|\bfx|\geqslant n-i+1}w(\bfx)\big(f(\bfx)-A(f)(\bfx)\big)=0 \qquad (i\in [n+1]).
\end{equation}
We observe that the system (\ref{eq:9s0df}) remains equivalent if we replace the inequality $|\bfx|\geqslant n-i+1$ with the equality. Using
(\ref{eq:98s7}), we then obtain
\begin{equation}\label{eq:98s7a}
\bigg(\sum_{j=i}^{n+1}c_j\bigg)\bigg(\sum_{|\bfx|= n-i+1}w(\bfx)\bigg)=\sum_{|\bfx|= n-i+1}w(\bfx)f(\bfx) \qquad (i\in [n+1])
\end{equation}
We finally obtain the result by using (\ref{eq:76dfsd}) and subtracting equation $i+1$ from equation $i$.
\end{proof}

We now provide alternative expressions for $A(f)$ as a shifted $L$-statistic function and symmetric multilinear polynomials. Observing first
that $\overline{f}(0)=f(\mathbf{0})$ and then using (\ref{eq:98s7}) and (\ref{eq:9s0df2}) for $i=n+1$, we obtain
\begin{equation}\label{eq:sdf789a}
A(f) ~=~ f(\mathbf{0})+\sum_{j=1}^n c_j\, \mathrm{os}_j ~=~ \langle f,1\rangle + \sum_{j=1}^n c_j
\big(\mathrm{os}_j-\langle\mathrm{os}_j,1\rangle\big),
\end{equation}
where $c_j = \overline{f}(n-j+1)-\overline{f}(n-j)$ for every $j\in [n]$. Then, using (\ref{eq:76dfsd}), (\ref{eq:98s7}), and (\ref{eq:98s7a}),
we obtain
\begin{equation}\label{eq:useful}
A(f)(S) ~=~ \sum_{j=n-|S|+1}^{n+1}c_j ~=~ E(f(\bfx)\mid |\bfx|=|S|) ~=~ \overline{f}(|S|)\qquad (S\subseteq [n])
\end{equation}
so that by (\ref{eq:pBfPF}) we obtain immediately
$$
A(f)(\bfx) ~=~ \sum_{S\subseteq [n]} \overline{f}(|S|)\,\prod_{i\in S}x_i\,\prod_{i\in [n]\setminus S}(1-x_i) ~=~ \sum_{S\subseteq [n]}
\Delta_k^{|S|}\, \overline{f}(k)|_{k=0}\,\prod_{i\in S}x_i\, ,
$$
where $\Delta_k^{s}\, \overline{f}(k)|_{k=0}=\sum_{t=0}^s{s\choose t}(-1)^{s-t}\,\overline{f}(t)$.

We now examine the effect of a permutation of the variables of $f$ on the symmetric approximation $A(f)$. Let $S_n$ denote the symmetric group
on $[n]$. A permutation $\pi\in S_n$ acts on a pseudo-Boolean function $f\in F(\B^n)$ by
$\pi(f)(x_1,\ldots,x_n)=f(x_{\pi(1)},\ldots,x_{\pi(n)})$. A permutation $\pi\in S_n$ is said to be a \emph{symmetry} of $f\in F(\B^n)$ if
$\pi(f)=f$.

\begin{proposition}\label{prop:sdg79}
If $\pi\in S_n$ is a symmetry of the weight function $w$, then for every $f\in F(\B^n)$ we have $A(\pi(f))=A(f)$ and
$\|\pi(f)-A(f)\|=\|f-A(f)\|$.
\end{proposition}

\begin{proof}
If $\pi$ is a symmetry of $w$, then clearly it is an isometry of $F(\B^n)$, that is, $\langle\pi(f),\pi(g)\rangle=\langle f,g\rangle$. Now, if
$g\in F_S(\B^n)$, then by (\ref{eq:9s0df2}), we have $\langle\pi(f),g\rangle=\langle\pi(f),\pi(g)\rangle=\langle f,g\rangle=\langle
A(f),g\rangle$, which shows that $A(\pi(f))=A(f)$. We prove the second equality similarly since $\|\pi(f)-A(f)\|^2=\|\pi(f)-\pi(A(f))\|^2$.
\end{proof}

With any pseudo-Boolean function $f\in F(\B^n)$, we can associate the symmetric function $ \mathrm{Sym}(f)=\frac{1}{n!}\,\sum_{\pi\in
S_n}\pi(f). $ We then have the following result.

\begin{corollary}\label{cor:sdg79}
If the weight function $w$ is symmetric, then for every $f\in F(\B^n)$ we have $\mathrm{Sym}(f)=A(\mathrm{Sym}(f))=A(f)$.
\end{corollary}

\begin{proof}
The first equality follows from the symmetry of $\mathrm{Sym}(f)$. The second one follows from Proposition~\ref{prop:sdg79} and the linearity of
the projector $A$.
\end{proof}

We end this section by analyzing the effect of dualization of $f$ on the symmetric approximation $A(f)$. The \emph{dual} of a function $f\in
F(\B^n)$ is the function $f^d\in F(\B^n)$ defined by $f^d(\bfx)=f(\mathbf{0})+f(\mathbf{1})-f(\mathbf{1}_{[n]}-\bfx)$.

\begin{proposition}\label{prop:adsf711}
If the weight function $w$ satisfies $w(\mathbf{1}_{[n]}-\bfx)=w(\bfx)$ for all $\bfx\in\B^n$, then for every $f\in F(\B^n)$ we have
$A(f^d)=A(f)^d$.
\end{proposition}

\begin{proof}
By (\ref{eq:9s0df2}) and (\ref{eq:useful}), we have $A(f)(\mathbf{1})-f(\mathbf{1})=A(f)(\mathbf{0})-f(\mathbf{0})=\langle A(f),1\rangle
-\langle f,1\rangle=0$. From these equalities, we obtain
$$
\langle f^d-A(f)^d,g^d\rangle=\langle (f-A(f))^d,g^d\rangle=\langle f-A(f),g\rangle=0
$$
for every function $g\in F_S(\B^n)$. The result then follows.
\end{proof}

%---------------------------------------------------------------------------------------------- Section 3
\section{Influence of the $k$th smallest variable}

Following Hammer and Holzman's approach \cite{HamHol92}, to measure the influence of the $k$th smallest variable $x_{(k)}$ on a pseudo-Boolean
function $f\in F(\B^n)$, it is natural to define an index $I\colon F(\B^n)\times [n]\to\R$ as $I(f,k)=c_k$, where $c_k$ is defined in
Theorem~\ref{thm:as897}.\footnote{We observe that, by definition, this index remains invariant under normalization of $w$.}

\begin{definition}\label{de:dfdgf78}
Let $I\colon F(\B^n)\times [n]\to\R$ be defined as $I(f,k)=\overline{f}(n-k+1)-\overline{f}(n-k)$.
\end{definition}

Thus we have defined an influence index from an elementary approximation (projection) problem. Conversely, the following result shows that
$A(f)$ is the unique function of $F_S(\B^n)$ that preserves the average value and the influence index of $f$.

\begin{proposition}
A function $g\in F_S(\B^n)$ is the best symmetric approximation of $f\in F(\B^n)$ if and only if $\langle f,1\rangle=\langle g,1\rangle$ and
$I(f,k)=I(g,k)$ for all $k\in [n]$.
\end{proposition}

\begin{proof}
The necessity is trivial (use Eq.~(\ref{eq:9s0df2}) for $i=n+1$). To prove the sufficiency, observe that any $g\in F_S(\B^n)$ satisfying the
assumptions of the proposition is of the form
$$
g ~=~ g(\mathbf{0})+\sum_{j=1}^n I(g,j)\,\mathrm{os}_j ~=~ g(\mathbf{0})+\sum_{j=1}^n I(f,j)\,\mathrm{os}_j\, .
$$
We then have $g(\mathbf{0})+\sum_{j=1}^n I(f,j)\,\langle\mathrm{os}_j,1\rangle=\langle g,1\rangle=\langle f,1\rangle$. Using (\ref{eq:sdf789a}),
we finally obtain $g=A(f)$.
\end{proof}

The next proposition reassembles several properties of the index $I(f,k)$. These properties follow easily from the definition of the index and
the properties of the approximations.

\begin{proposition}
Let $k\in [n]$ and let $w\colon\B^n\to {\left]{0},{\infty}\right[}$ be a weight function. Then
\begin{enumerate}
\item[$(i)$] The map $f\mapsto I(f,k)$ is linear.

\item[$(ii)$] If $\pi$ is a symmetry of $w$, then $I(\pi(f),k)=I(f,k)$ for every $f\in F(\B^n)$.

\item[$(iii)$] If $w$ is symmetric, then $I(f,k)=I(\mathrm{Sym}(f),k)$ for every $f\in F(\B^n)$.

\item[$(iv)$] If $w$ satisfies $w(\mathbf{1}_{[n]}-\bfx)=w(\bfx)$ for all $\bfx\in\B^n$, then $I(f^d,k)=I(f,n-k+1)$ for every $f\in F(\B^n)$.

\item[$(v)$] We have $\sum_{j=1}^n I(f,j)=f(\mathbf{1})-f(\mathbf{0})$.
\end{enumerate}
\end{proposition}

It is a well-known fact of linear algebra that a linear map on a finite dimensional inner product space can be expressed as an inner product
with a fixed vector. The next proposition gives the explicit form of such a vector for $I({\,\boldsymbol{\cdot}\,},k)$. To this extent, for
every $k\in [n]$ we introduce the function $g_k\colon\B^n\to\R$ as $g_k(\bfx)=\Delta_k(d_k\Delta_k\mathrm{os}_{k-1})$, where
$d_k=-1/\sum_{|x|=n-k+1}w(\bfx)$.

\begin{proposition}\label{prop:sd78dsad}
For every $f\in F(\B^n)$ and every $k\in [n]$, we have $I(f,k)=\langle f,g_k\rangle$.
\end{proposition}

\begin{proof}
We have $d_{k+1}\langle f,\Delta_k\mathrm{os}_{k}\rangle=d_{k+1}\sum_{|\bfx|=n-k}w(\bfx)f(\bfx)=-\overline{f}(n-k)$, which leads immediately to
the result.
\end{proof}

Proposition~\ref{prop:sd78dsad} shows that the index $I(f,k)$ is the covariance of the random variables $f$ and $g_k$. Indeed, we have
$I(f,k)=E(f\, g_k)=\mathrm{cov}(f,g_k)+E(f)\, E(g_k)$, where $E(g_k)=\langle 1,g_k\rangle = I(1,k)=0$. From the usual interpretation of the
concept of covariance, we see that an element $\bfx\in\B^n$ makes a positive contribution to $I(f,k)$ whenever the values of $f(\bfx)-E(f)$ and
$g_k(\bfx)-E(g_k)=g_k(\bfx)$ have the same sign. Note that $g_k(\bfx)$ is positive whenever $x_{(k)}$ is greater than the value
$(d_{k+1}x_{(k+1)}+d_kx_{(k-1)})/(d_{k+1}+d_k)$, which lies in the range of $x_{(k)}$ when the other order statistics are fixed at $\bfx$.

%---------------------------------------------------------------------------------------------- Section 4
\section{Two special cases: Cardinality index and system signature}

We now show that the \emph{cardinality index} and \emph{system signature} are particular instances of our influence index.

\subsection{The cardinality index of a fuzzy measure}

A fuzzy measure on the finite set $X=[n]$ is a nondecreasing set function $\mu\colon 2^X\to [0,1]$ satisfying the boundary conditions
$\mu(\varnothing)=0$ and $\mu(X)=1$. For any subset $S\subseteq X$, the number $\mu(S)$ can be interpreted as the certitude that we have that a
variable will take on its value in the set $S\subseteq X$. In this context, Yager \cite{Yag02} introduced the \emph{cardinality index}
associated with a fuzzy measure $\mu$ as the $n$-tuple $(C_0,\ldots,C_{n-1})$ where $C_k$ is the average gain in certitude that we obtain by adding
an arbitrary element to an arbitrary $k$-subset, that is,
$$%\begin{equation}\label{Ck}
C_k=\frac{1}{(n-k){n\choose k}}~\sum_{|S|=k}~\sum_{x\notin S}\big(\mu(S\cup\{x\})-\mu(S)\big).
$$%\end{equation}
We observe that this expression, which resembles the Banzhaf value (\ref{eq:sa8ffd}), could be used in cooperative game theory to measure the
marginal contribution of an additional player to a $k$-coalition. It is also clear that this index can be written as
$$%\begin{equation}\label{Ckprime}
C_k=\frac{1}{{n\choose {k+1}}}\sum_{|S|=k+1}\mu(S)-\frac{1}{{n\choose k}}\sum_{|S|=k}\mu(S)\, ,
$$%\end{equation}
which shows that we have $C_k=I(\mu,n-k)=I(\mu^d,k+1)$ in the special case when the weight function $w$ defining the index $I$ is symmetric.

\subsection{System signatures in engineering reliability}
\label{sec:2}

Consider a \emph{system} consisting of $n$ interconnected components. When the components have continuous and i.i.d.\ lifetimes
$X_1,\ldots,X_n$, the \emph{signature} of the system is defined as the $n$-tuple $(s_1,\ldots,s_n)\in [0,1]^n$ with $s_k=\Pr(T=X_{(k)})$, where
$T$ denotes the system lifetime. That is, $s_k$ is the probability that the $k$th failure causes the system to fail (for a recent reference, see
\cite{Sam07}). It was proved \cite{Bol01} that
$$
s_k=\frac{1}{{n\choose n-k+1}}\sum_{|\bfx|=n-k+1}\phi(\bfx)-\frac{1}{{n\choose n-k}}\sum_{|\bfx|=n-k}\phi(\bfx)\, ,
$$
where $\phi\colon\B^n\to\B$ is the structure function of the system. Thus, in view of Definitions~\ref{def:vf} and \ref{de:dfdgf78}, we have
$s_k=I(\phi,k)$ in the special case where the weight function $w$ is symmetric. Interestingly, the identity $s_k=I(\phi,k)$ still holds in the
non-i.i.d.\ case if we define the weight function $w$ as the (non-normalized) \emph{relative quality function}
$$
w(S)=\Pr\Big(\max_{i\in [n]\setminus S}X_i<\min_{j\in S}X_j\Big)
$$
for which we have $\sum_{|\bfx|=s}w(\bfx)=1$ for all $s\in [n]$ (see \cite{MarMat11}). Therefore $s_k$ can be obtained from a weighted least
squares approximation problem of the structure function and can always be interpreted as the influence on the system of the component that has
the $k$th smallest lifetime.

%---------------------------------------------------------------------------------------------- Acknowledgments
\section*{Acknowledgments}

This research was supported by the internal research project F1R-MTH-PUL-09MRDO of the University of Luxembourg.

\end{document}